\documentclass{article}

\usepackage[final]{neurips_2025}

\usepackage[utf8]{inputenc} 
\usepackage[T1]{fontenc}    
\usepackage{hyperref}       
\usepackage{url}            
\usepackage{booktabs}       
\usepackage{amsfonts}       
\usepackage{amssymb}
\usepackage{nicefrac}       
\usepackage{microtype}      
\usepackage{xcolor}         
\usepackage{mathtools}
\usepackage{amsthm, amsmath}
\usepackage{tikz}
\usetikzlibrary{math}
\usetikzlibrary{shapes.misc}
\usepackage{pgfplots}
\usetikzlibrary{calc}
\usetikzlibrary{arrows.meta,backgrounds,fit,positioning}
\usepgfplotslibrary{patchplots}
\usepgfplotslibrary{fillbetween}
\pgfplotsset{%
    layers/standard/.define layer set={%
        background,axis background,axis grid,axis ticks,axis lines,axis tick labels,pre main,main,axis descriptions,axis foreground%
    }{
        grid style={/pgfplots/on layer=axis grid},%
        tick style={/pgfplots/on layer=axis ticks},%
        axis line style={/pgfplots/on layer=axis lines},%
        label style={/pgfplots/on layer=axis descriptions},%
        legend style={/pgfplots/on layer=axis descriptions},%
        title style={/pgfplots/on layer=axis descriptions},%
        colorbar style={/pgfplots/on layer=axis descriptions},%
        ticklabel style={/pgfplots/on layer=axis tick labels},%
        axis background@ style={/pgfplots/on layer=axis background},%
        3d box foreground style={/pgfplots/on layer=axis foreground},%
    },
}
\pgfplotsset{compat=1.18}
\usepackage{graphicx}
\usepackage{subcaption}

\newtheoremstyle{break}%
{}{}%
{\itshape}{}%
{\bfseries}{.}%
{\newline}{}

\theoremstyle{break}
\newtheorem{breaktheorem}{Theorem}
\newtheorem{breakproposition}[breaktheorem]{Proposition}

\theoremstyle{plain}

\newtheorem{proposition}[breaktheorem]{Proposition}
\newtheorem{lemma}[breaktheorem]{Lemma}
\newtheorem{cor}[breaktheorem]{Corollary}

\theoremstyle{definition}
\newtheorem{definition}[breaktheorem]{Definition}
\newtheorem{example}[breaktheorem]{Example}
\newtheorem{remark}[breaktheorem]{Remark}

\theoremstyle{remark}

\renewcommand{\H}{\mathcal{H}}

\newcommand{\id}{\text{id}}
\newcommand{\R}{\mathbb{R}}

\title{Non-Singularity of the Gradient Descent Map for Neural Networks with
Piecewise Analytic Activations}

\author{%
	Alexandru Cr\u{a}ciun \\
	Technical University of Munich\\
	\texttt{a.craciun@tum.de} \\
	\And
	Debarghya Ghoshdastidar\\
	Technical University of Munich\\
	Munich Data Science Institute\\
	Munich Center for Machine Learning\\
	\texttt{ghoshdas@cit.tum.de} \\
}

\begin{document}

\maketitle

\begin{abstract}
	The theory of training deep networks has become a central question of modern
	machine learning and has inspired many practical advancements. In
	particular, the gradient descent (GD) optimization algorithm has been
	extensively studied in recent years. A key assumption about GD has appeared
	in several recent works: the \emph{GD map is non-singular} --- it preserves
	sets of measure zero under preimages. Crucially, this assumption has been
	used to prove that GD avoids saddle points and maxima, and to establish the
	existence of a computable quantity that determines the convergence to global
	minima (both for GD and stochastic GD). However, the current literature
	either assumes the non-singularity of the GD map or imposes restrictive
	assumptions, such as Lipschitz smoothness of the loss (for example,
	Lipschitzness does not hold for deep ReLU networks with the cross-entropy
	loss) and restricts the analysis to GD with small step-sizes. In this paper,
	we investigate the neural network map as a function on the space of weights
	and biases. We also prove, for the first time, the non-singularity of the
	gradient descent (GD) map on the loss landscape of realistic neural network
	architectures (with fully connected, convolutional, or softmax attention
	layers) and piecewise analytic activations (which includes sigmoid, ReLU,
	leaky ReLU, etc.) for almost all step-sizes. Our work significantly
	extends the existing results on the convergence of GD and SGD by
	guaranteeing that they apply to practical neural network settings and has
	the potential to unlock further exploration of learning dynamics.
\end{abstract}

\section{Introduction}
\label{section:introduction}

Training deep neural networks involves optimizing highly non-convex loss
functions over high-dimensional parameter spaces, a process that poses
significant theoretical and practical challenges. Among these are the risks of
gradient descent (GD) converging to saddle points or poor local minima, which
could hinder the performance of trained models, especially when solutions
associated with worst-case saddle points are significantly worse than those
associated with worst-case local minima \citep{dauphin2014, pascanu2014}. Recent
theoretical advances have shed light on the dynamics of GD in such landscapes.
One key result demonstrates that GD avoids converging to saddle points or
maxima for almost all initializations, provided certain
conditions are satisfied \citep{lee2019}. This supports empirical observations
that GD tends to converge to local minima rather than get trapped in poor
solutions, at least in the idealized setting in \citet{lee2019}. Another line of
research has explored the stability of minima, introducing Lyapunov
exponent-like quantities that characterize whether GD converges to a given
minimum from nearby initializations \citep{chemnitz2024}. These findings offer critical
insights into why optimizing using GD often succeeds in practice despite the
complexity and non-convex nature of the neural network loss.

\begin{figure}[ht!]
	\begin{center}
		\begin{tikzpicture}[xscale=1.7,yscale=0.4, ultra thick]
			\tikzmath{
				real \x1, \y1, \x2, \y2, \xc, \yc;
				\x1 = 2.1;
				\y1 = \x1^2;
				\x2 = 2.7;
				\y2 = \x2^2;
				\xc = 0;
				\yc = 8;
				\l = -0.3;
				\r = 3.0;
				function loss(\x){
					if \x < -2 then { 
						return (\x)^2; 
					} else {
						if \x < 2 then { 
							return {0.5*(\x)^4-3*(\x)^2+8}; 
						} else {
							return (\x)^2; 
						}; 
					};
				};
				{
					\draw[domain=\l:\r,color=blue, samples=100] plot (\x, {loss(\x)});
					\draw[domain=\x1:\x2,color=red]	plot (\x, {loss(\x)});
				};
			}

			\node (e) at (0.5, 5.0) {$\textcolor{red}{\eta=0.5}$};
			\draw[color=red, fill=red] (\xc, \yc) ellipse (1pt and 4pt);
			\draw[color=red, fill=red] (\x1, \y1) ellipse (1pt and 4pt);
			\draw[color=red, ->] (\x1, \y1) -- (\xc + 0.1, \yc);
			\node (a) at (1.5, 6.4) {$\textcolor{red}{G_\eta}$};
			\draw[color=red, fill=red] (\x2, \y2) ellipse (1pt and 4pt);
			\draw[color=red, ->] (\x2, \y2) -- (\xc + 0.1, \yc);
		\end{tikzpicture}
		\qquad
		\begin{tikzpicture}[xscale=1.7,yscale=0.4, ultra thick]
			\tikzmath{
				\l = -0.3;
				\r = 3.0;
				function loss(\x){
					if \x < -2 then { 
						return {(\x)^2}; 
					} else {
						if \x < 2 then { 
							return {0.5*(\x)^4-3*(\x)^2+8}; 
						} else {
							return {(\x)^2}; 
						}; 
					};
				};
				{\draw[domain=\l:\r,color=blue, samples=100] plot (\x, {loss(\x)});};
				function grad(\x, \e){
					if \x < -2 then { 
						return \x - 2*\e*\x; 
					} else {
						if \x < 2 then { 
							return {\x - \e*(2*(\x)^3-6*(\x))}; 
						} else {
							return \x - 2*\e*\x; 
						}; 
					};
				};
				real \x1, \x2;
				\x1 = 2.1;
				\x2 = 2.7;
				\et = 0.25;
				\xo = \x1;
				\yo = loss(\xo);
				\xn = 0.0;
				\yn = 0.0;
				{ \draw[domain=\x1:\x2,color=red]	plot (\x, {loss(\x)}); }
				{ \draw[color=red, fill=red] (\xo, \yo) ellipse (1pt and 4pt); };
				for \i in {1,...,1} {
					\xn = grad(\xo, \et);
					\yn = loss(\xn);
					{ \draw[color=red, ->] (\xo, \yo) -- (\xn, \yn); };
					\xo = \xn;
					\yo = \yn;
				};
				\xo = \x2;
				\yo = loss(\xo);
				\xn = 0.0;
				\yn = 0.0;
				{ \draw[color=red, fill=red] (\xo, \yo) ellipse (1pt and 4pt); };
				for \i in {1,...,1} {
					\xn = grad(\xo, \et);
					\yn = loss(\xn);
					{ \draw[color=red, ->] (\xo, \yo) -- (\xn, \yn); };
					\xo = \xn;
					\yo = \yn;
				};
				\x1 = grad(\x1, \et);
				\x2 = grad(\x2, \et);
				{ \draw[domain=\x1:\x2,color=red]	plot (\x, {loss(\x)}); };
			}
			\node (e) at (0.5, 5.0) {$\textcolor{red}{\eta=0.25}$};
			\node (a) at (1.7, 6.4) {$\textcolor{red}{G_\eta}$};
		\end{tikzpicture}
	\end{center}
	\caption{Illustrating the difference between singular and non-singular GD
		maps (different step-sizes). The loss used is $ L(\theta) =
		0.5\theta^4-3\theta^2+8 \text{ if } -2 \leq \theta \leq 2; \theta^2
		\text{ otherwise}$. Left: for $ \eta=0.5 $, the GD map is
		\emph{singular}. The red interval gets mapped to a single point after
		one iteration. Right: for $ \eta = 0.25 $, the GD map is
		\emph{non-singular}. The same interval is mapped not to a point, but to
		an interval.
	}\label{fig:ill}
\end{figure}
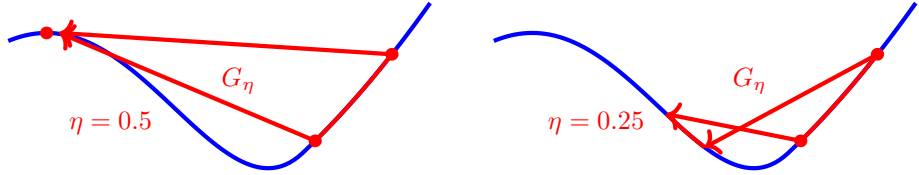

Both of the analyses rely on the fact that the GD map --- defined as $
G_\eta(\theta) = \theta - \eta \nabla L(\theta) $, where $ L $ is the loss
function and $ \eta > 0$ is the step-size --- is non-singular. A map $ G $ is
\emph{non-singular} if the preimage of any set of measure zero under $ G $ also
has measure zero; maps not having this property are called \emph{singular}.
Non-singularity of the GD map ensures that pathological behaviors (e.g.,
convergence to undesirable points) occur only on negligible sets. An example of
the difference between the two concepts can be seen in Figure~\ref{fig:ill},
which illustrates the effect of performing one optimization step. While
non-singularity is relevant to many theoretical guarantees (such as the two
works mentioned above), its validity in the context of neural networks has not
been put on a solid footing.  If the neural network uses an analytic activation
function, it can be shown using only the standard tools of analysis that the
gradient descent map is non-singular. However, if one employs a strictly
piecewise analytic activation function like ReLU, the standard tools no longer
work and a novel approach is needed.

In this paper, we prove that the (stochastic) GD map is non-singular for almost all
step-sizes $ \eta $ for neural networks using piecewise analytic activations
and a number of different architectures, for example using fully connected,
convolutional, or softmax attention layers, thereby validating the assumptions
underpinning prior optimization theories in the context of practical neural
network training. Applying the standard tools no longer works since the
composition of two almost everywhere analytic functions need not be almost
everywhere analytic itself (cf.\ Remark~\ref{ex:crf}). However, we can make use
of the layered nature of a neural network to prove that an analogue to the
chain rule holds for the neural network function
(Proposition~\ref{prop:ff-fcaea}). This result explains why the function is
well-behaved for all parameter values except for a negligible set. With this
insight, we extend the technique used to prove the non-singularity of the GD map
in the case of neural networks with \emph{strictly} analytic activations to
networks using \emph{piecewise} analytic activations (such as ReLU).  Thus, our
main result is:

\begin{breaktheorem}[\bf Stochastic Gradient Descent Map for Neural Networks is Non-Singular]
	\label{thm:main}
	Consider a deep neural network that consists of fully connected,
	convolutional, or attention layers and let the non-linear
	activations in the layers be piecewise analytic. Additionally, fix any data
	and any analytic loss function. Then, for almost all step-sizes $ \eta $,
	any (S)GD map $G_{\eta}$ is non-singular.
\end{breaktheorem}

This contribution bridges a significant gap between theoretical optimization
literature and practical neural network training. By establishing the
non-singularity of the GD map, we extend the applicability of results on
saddle-point and maxima avoidance \citep{lee2016, panageas2016, lee2019,
cheridito2022, panageas2025} and on the stability of minima \citep{wu2018,
ma2021, ahn2022, chemnitz2024} to realistic deep learning settings, offering a
rigorous explanation for the empirical success of GD. Of particular importance
is the fact that the GD map is non-singular irrespective of the number of data
points, hence theoretical results on the optimization properties of stochastic
GD (SGD), such as those in~\cite{chemnitz2024} can also be extended to realistic
deep learning scenarios using our framework.

\subsection{Related Works}
\label{section:related_works}
Understanding the optimization dynamics of neural networks has been a central
focus of machine learning research, blending classical mathematical tools with
domain-specific insights. A well-known result in this area is Rademacher’s
Theorem \citep{rockafellar1998}, which states that locally Lipschitz functions
are differentiable almost everywhere. While this provides a general framework
for studying smoothness, it does not directly address the unique structure of
neural network losses, which are influenced by architectures and
activation functions, nor does it provide more refined knowledge on the GD
map (such as non-singularity), which is usually assumed in studying
optimization.

Recent studies have focused on the piecewise nature of activation functions, such
as ReLU, to better to understand how the layered structure of a deep neural
network tessellates the input space into affine-linear regions as the parameters
vary \citep{montufar2014, balestriero2019, balestriero2020}. Software to
efficiently visualize these regions has also been developed \citep{humayun2023}.
These works have shed light on the geometric properties of the network function
from the input space for fixed parameters: they estimate the number of affine
linear regions of the resulting function and how this changes as the weights and
biases of the network change. However, from an optimization point of view, one
is interested in understanding the loss landscape. That is, for fixed data, one
wants to know how the piecewise nature of the activation
function splits the parameter space into different regions, separated by
boundaries where the loss function is not differentiable.

In optimization theory, \citet{panageas2016, lee2019} showed that first-order
methods like GD almost always avoid saddle points under strict assumptions such
as restricting the step-size to small values and requiring the loss function to
be Lipschitz smooth. These are used to prove that the GD map is non-singular,
which in turn is used to show that the set of initializations that converge to a
saddle point has measure zero (if the set of saddle points has measure zero
itself). \citet{khromov2023} have empirically investigated
how the Lipschitzness of neural networks as functions from the input space
changes during training and present bounds for a number of architectures and
datasets.  However, it is the Lipschitzness of the map from the parameter space
which is needed to prove that GD avoids saddles and maxima. This condition is
almost never true in neural networks (for example, it does not hold for deep
ReLU networks with cross entropy as loss).

Ideas relating to a heuristic notion of stability of minima for GD and SGD have
been explored in \citet{wu2018, ma2021}, where they conducted an investigation
of how the batch-size and the step-size influence stability. \citet{wu2018}
train a model using GD until convergence and then switch to SGD, which, most of
the time, converges to a different minimum (that generalizes better).
Theoretical works such as \citet{ahn2022} and \citet{chemnitz2024} start by
assuming that the GD map is non-singular and use this to show that a notion of
stability can be defined for minima. They conclude that stable minima attract
nearby initializations, while unstable ones repel almost all starting points.
The main contribution of \citet{chemnitz2024} is to extend the analysis of
stability for GD to the stochastic setting, which is a non-trivial procedure.
We mention that all the works referenced in this paragraph treat generic
parametric models, not specializing to neural networks.  Their assumption, that
the gradient descent map is non-singular, while plausible, lacks rigorous
justification in the neural network context, especially if the activation
function is strictly piecewise analytic (such as ReLU, GELU, leaky ReLU, etc.).
We also mention that stability has been empirically observed to be related to
generalization ability \citep{hochreiter1997}.

Closest to this paper is the work of \citet{riekert2022, riekert2022conv,
riekert2023}. While their work is more concerned with the convergence properties
of gradient descent and gradient flow, they show that the loss function is
continuously differentiable almost everywhere for deep neural networks using
only ReLU as activation function under the assumption that the underlying
function generating the data is polynomial. Our work expands on their results by
allowing for any piecewise analytic activation function (including, but not
limited to, sigmoid or ReLU), a broader class of network architectures (including
convolutional and networks using attention layers), as well as a more general
learning setting (we do not impose any restrictions on the data-generating
process).

Our work builds on and extends these efforts by providing a formal proof that
the gradient descent map is non-singular for neural networks with piecewise
analytic activations, such as sigmoid, tanh, or ReLU. We show that the loss
function is analytic almost everywhere and that the GD map
preserves the measure zero property under preimages for almost all step-sizes. This
result not only validates the assumptions in prior works but also establishes a
stronger theoretical foundation for understanding optimization in deep learning,
distinguishing our contribution from earlier heuristic analyses.

\subsection{Organization of the Paper}
In Section~\ref{section:setting} we introduce the optimization task,  different
neural network architectures, and the (stochastic) gradient descent algorithm.
In Section~\ref{section:loss_func_analytic} we prove that the empirical loss
function is almost everywhere analytic for neural networks with piecewise
analytic activations. Next, in Section~\ref{section:gd_map_non_sing} we
establish the non-singularity of the gradient descent map, and show why some
values of the step-size should be avoided. Finally, in
Section~\ref{section:discussion} we discuss the implications of this work,
provide further illustrations and examples, and highlight directions for future
research.

\section{Setting}
\label{section:setting}

\par{\bf Notation and Conventions.}
Throughout this paper, we will only use the Lebesgue measure. Additionally,
we use the abbreviation a.e.\ to mean almost everywhere.  For a function $
f:\R^m \to \R^n $ we denote by $ Z(f) = \{ x \in \R^m | f(x) = 0 \} $ the set of
points where $ f $ is zero and if $ U \subset \R^m $ is a set, we denote by $
\partial U $ its topological boundary. For a finite set $ I $, $ |I| $ denotes the
number of elements in $ I $.

\par{\bf Supervised Learning.}
In the following, we consider a generic supervised learning problem. Let $ F:
\R^{n_{\theta}}\times \R^{n_0} \to \R^{n_D} $ be a \emph{parametric} model,
that is, we consider the first space in the Cartesian product of the domain of $ F $
to be parameters. For each parameter $ \theta \in \R^{n_\theta} $ we get a
model that associates to an input $ x \in \R^{n_0} $ an output vector $
f_{\theta} (x) = F(\theta, x) \in \R^{n_D} $. 

Given training data $ \left( (x_i, y_i)\right)_{i=1}^m \subset \R^{n_0} \times
\R^{n_D} $ and a loss function $ l:\R^{n_D}\times \R^{n_D} \to \R $, the goal is
to find a set of parameters $ \theta^* \in \R^{n_\theta} $ that minimizes the
empirical loss
\begin{equation} 
	\label{eq:loss}
	L(\theta) = \frac{1}{m} \sum_{i=1}^m l(y_i, F(\theta, x_i)).
\end{equation}

\par{\bf Neural Networks.}
Given a positive integer $ D $ called the \emph{depth} of the neural network, a
sequence $ N = (n_0, \ldots, n_D) $ of positive integers called the
\emph{widths}, and a sequence of piecewise analytic functions $ \Sigma =
\{\sigma_d:\R \to \R\}_{d=1}^{D} $ called
the \emph{activation functions}, we denote by $ \H_{H,N,\Sigma} $ the hypothesis
class of \emph{feed-forward fully connected neural networks} (FF-FC) with a fixed
depth and widths, and with piecewise analytic activation functions. These 
networks are the simplest and serve as a starting point for the
more complicated architectures we will deal with.

We now describe how such a network works. An FF-FC network $ f_\theta $ is given as
\[ 
	f_\theta(x) = \left(f^{ \left(D\right)} \circ \cdots \circ f^{
	\left(1\right)}\right)(x).
\]
Here $ \theta = (W_1, b_1, \ldots, W_D, b_D)$ denotes the collection of
parameters from each layer, where $ b_i \in \R^{n_i} $ and $ W_i \in \R^{n_i
\times n_{i-1}} $ are called the \emph{bias} and \emph{weights} for that layer, respectively; we
let $ n_\theta = n_0n_1 + \ldots + n_{D-1}n_D + n_1 + \ldots + n_D$ denote the
total number of parameters.

The $ d $-th network layer $ f^{ \left(d\right)} $ takes as input a vector $ z^{
 \left(d-1\right)} $, the output of the previous network layer, and acts as follows:
 \[ 
 	f^{ \left(d\right)} (z^{ \left(d-1\right)}) = \sigma_d (W_d z^{\left(d-1\right)} + b_d),
 \]
where we apply the activation function $ \sigma_d $ entry-wise.

In the rest of the paper, we will also be referring to convolutional or softmax
attention layers. Such layers have been developed to solve specific tasks such
as image recognition and natural language interpretation. A convolutional layer
performs a convolution on the output from the previous layer using some kernel.
Attention layers use a query, key, value system in order to find relevant tokens
and better embeddings for sequential data. Deeper network architectures can be
obtained by stacking together different types of layers.  This setting is
applicable to the transformer architecture~\citep{vaswani2017} that interleaves
convolutional and attention layers, followed by a fully connected network.
Please refer to Appendix~\ref{section:architectures} for more details on
different types of architectures and specialized layers.

\par{\bf Bias is zero.}
We mention here that for the rest of this paper, we set the bias terms in each
layer equal to zero. This is done for clarity of exposition as dealing with the
bias terms can be done in a straightforward way (cf.\ Remark~\ref{rem:drop-bias}
in Appendix~\ref{section:proof}), but only makes the exposition harder to follow.

\par{\bf Stochastic GD (SGD).} Gradient descent (GD) is defined as the map $
G_\eta(\theta) = \theta - \eta\nabla L(\theta)$ for $ \eta > 0 $. For GD, we
initialize the parameters $ \theta_0 $ according to a probability distribution
on $ \R^{n_\theta} $ and update them using the rule $ \theta_{i+1} =
G_\eta(\theta_i) = \theta_i - \eta\nabla L(\theta_i)$, where the loss is the one
in~\eqref{eq:loss}.  For stochastic GD (SGD), at each step $ i $ we use only a
sub-sample of $ 0<n<m $ data points to compute the empirical loss.  Hence, the
new point is $ \theta_{i+1} = G_{\eta, i} (\theta_i) = \theta_i - \eta\nabla
L_i(\theta_i) $, where we randomly choose $ i_1, \ldots, i_n \in [m], i_j \neq
i_k $ if $ j \neq k $, and the empirical loss computed on a mini batch is 
\[ 
	L_i(\theta) = \frac{1}{n}\sum_{j=1}^n l(y_{i_j}, F(\theta, x_{i_j})).
\]

\section{Analyticity of the Loss Function}
\label{section:loss_func_analytic}
We now see that if the activation function of the network is analytic and the
loss function is analytic as well, the empirical loss function is also analytic.
This follows from the usual rules of calculus. It is no longer as obvious if
that is still the case when the activation function is only piecewise analytic.
The goal of this section is to show that the loss function for a neural network
using piecewise analytic activations is itself analytic on a ``big enough'' set.
To achieve this, we start by defining precisely what we mean by a piecewise
analytic function and extend this definition from univariate to multivariate
functions. Next, we highlight why a novel approach is needed and then show how
we can leverage the layered structure of the neural network to derive
analyticity guarantees. After the network function is properly understood, we
finish by showing that the empirical loss function is analytic on a ``big
enough'' set as well. Recall that a function $ f:\R^m \supset U \to \R^n $ is
analytic if for any point $ x_0 $ in its domain, it can be expressed as a
power series converging to $ f(x) $ for all $ x $ in a neighborhood of $ x_0 $.

\begin{definition}[\bf Piecewise and almost eveywhere analytic functions]
	\label{def:aea}
	We say a function $ f:\R \to \R $ is \emph{piecewise analytic} if there
	exists a strictly increasing sequence of real numbers $ \{x_i\}_{i\in
	\mathbb{Z}} $ such that $ f $ is analytic when restricted to any open
	interval $ (x_i, x_{i+1}) $. A function $ f:\R^m \to \R^n $ is
	\emph{analytic almost everywhere} if there exists an open set $ U \subset
	\R^m $ such that $ f|_U $ is analytic and the complement of $ U $ has
	measure zero. By $ S(f) $ we denote the points where $ f $ is not analytic
	and $ D(f) = \R^m \setminus S(f)$. 
\end{definition}

\begin{example}[\bf Some almost everywhere analytic functions]
	\leavevmode
	\begin{enumerate}
		\item 
			The sigmoid function $ f:x \mapsto \frac{1}{1+\exp(-x)} $ is piecewise
			analytic with $ D(f) = \R $.
		\item 
			The ReLU function $ f:x \mapsto \max\{0,x\} $ is also piecewise analytic.
			Take $ x_i = i $. We have that $ f|_{(x_i, x_{i+1})}(x) = 0 $
			if $ i < 0 $ and $ f|_{(x_i, x_{i+1})}(x) = x $ if $ i \ge 0 $. Here $ S(f) = \{0\} $.
	\end{enumerate}
\end{example}

\begin{remark}[\bf Non-differentiable points]
	In the previous definition, there may be multiple sets $ \{U_i\}_{i \in I} $
	on which $ f $ is analytic and whose complements have measure zero. Denoting
	by $ \mathcal{U} $ the collection of all such sets, we see that it is
	inductively ordered with respect to the usual inclusion of sets. An application of
	Zorn's lemma tells us that there exists a maximal such open set. It is not
	hard to check that it is also unique. This guarantees that the sets $ D(f) $
	and $ S(f) $ are well-defined.
\end{remark}

\begin{remark}[\bf Why a.e. analytic functions are problematic]
	\label{ex:crf}
	Observe that the chain rule does not hold for almost everywhere analytic functions.
	Whereas it is true that if $ f $ is a.e.\ analytic and $ g $ is analytic, we have
	that $ g\circ f $ is a.e.\ analytic, it is not true that $ f\circ g $ is
	a.e.\ analytic as the next (counter-)example shows. 
	Let $ h:\R \to \R $ be a nowhere differentiable function (e.g.\ the
	Weierstrass function) and define $ f:\R^2 \to \R $ as $ f(x, y) =
	h(x) $ if $ y = 0 $ and $ f(x,y) = 0$ otherwise. Since $ f $ is not analytic
	only on the abscissa ($ y=0 $), it is a.e.\ analytic. Taking $ g:\R \to \R^2 $ to be
	the function $ g(x) = (x, 0) $, we see that the
	composition $ f \circ g $ is equal to $ h $, hence it is not a.e.\ analytic.
\end{remark}

As the next proposition shows, it turns out that we can use the structure of a
neural network to derive an analogue to the chain rule for the function
obtained by stacking together layers in a neural network. This is important
since it will be used later to show that any gradient descent map obtained
from a neural network loss function is non-singular.

\begin{proposition}[\bf Analogue of chain rule for neural networks]
	\label{prop:ff-fcaea}
	Let $ D > 0 $ be a positive number, $ \{\sigma_i:\R^{n_i} \to
	\R^{n_i}\}_{i=1}^D $ be a collection of a.e. analytic maps and $ \alpha \in \R^{d_0} $ a
	vector. Then the map defined recursively by
	\begin{align*}
		f_D:\R^{n_1\times n_0}\times \cdots \times\R^{n_H\times n_{D-1}} &\to
		\R^{n_D}\\
		(W_1, \ldots, W_D) &\mapsto \sigma_D(W_Df_{D-1}(W_1, \ldots, W_{D-1})),
	\end{align*}
	with $ f_1(W_1) = \sigma_1(W_1\alpha)$ is analytic almost everywhere and the set $
	\partial Z(f_D) $ has measure zero.
\end{proposition}

\begin{proof}[Proof sketch]
	We use induction. For the base case, one uses the fact that the map $ m:W
	\mapsto W\alpha $ is a non-singular map (since it is a submersion). The
	points where $ f_1 $ is not differentiable have to lie in $ m^{-1} (S(\sigma_1))
	$, hence they are of measure zero. A similar argument works for the boundary
	of $ Z(f_1) $.

	For the induction step, we split the domain of $ f_{D-1} $: $ B(f_{D-1}) =
	\partial Z(f_{D-1}) \cup S(f_{D-1}) $, the ``bad'' points; $
	\text{int}(Z(f_{D-1})) $, the ``nice'' zeroes; and $ N(f_{D-1}) = \text{dom}
	(f_{D-1}) \setminus (B(f_{D-1}) \cup \text{int}(Z(f_{D-1}))) $, consisting
	of all ``nice'' non-zero points. The points in $ B(f_{D-1}) $ have measure
	zero, thus they can be neglected. One can then use the chain rule and
	properties of matrix multiplication to guarantee the analyticity of $ f_D $
	in a big enough region. A priori, this region might not be equal to $ D(f_D)
	$, but it is big enough to guarantee that $ S(f_D) $ has measure zero,
	showing that $ f_D$ is a.e.\ analytic.
\end{proof}

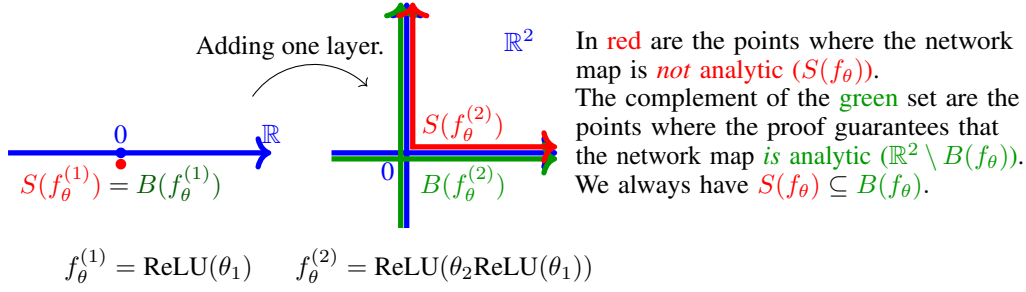
\begin{figure}[ht!]
	\centering
	\begin{tikzpicture}[line width=2]
		\draw[color=blue, ->] (-1.5, 0) -- (2,0);
		\draw[color=blue, fill=blue, radius=1pt] (0, 0) circle;
		\node at (0, 0.25) {\textcolor{blue}{0}};
		\node at (2, 0.25) {\textcolor{blue}{$\R$}};

		\draw[color=red, fill=red, radius=1pt] (0, -0.15) circle;
		\node at (0, -0.4) {\textcolor{red}{$S(f^{(1)}_\theta)$} $ = $
		\textcolor{green!60!black!60!black}{$B(f^{(1)}_\theta)$}};

		\node at (0.5, -1.5) {\textcolor{black}{$f^{(1)}_\theta = \text{ReLU}(\theta_1)$}};

		\node at (2.25, 1.4) {Adding one layer.};
		\draw[black, thin, ->] (1.75, 0.75) arc (145:45:1);
	
		\begin{scope}[shift={(-0.2,0.0)}]
			\draw[color=blue, ->] (3, 0) -- (6,0);
			\draw[color=blue, ->] (4, -1) -- (4,2);
			\node at (3.75, -0.25) {\textcolor{blue}{0}};
			\node at (5.5, 1.5) {\textcolor{blue}{$\R^2$}};

			\draw[color=green!60!black, ->] (3, -0.08) -- (6, -0.08);
			\draw[color=green!60!black, ->] (3.92, -1) -- (3.92,2);
			\node at (4.75, -0.4) {\textcolor{green!60!black}{$B(f^{(2)}_\theta)$}};

			\draw[color=red, ->] (4.05, 0.08) -- (6, 0.08);
			\draw[color=red, ->] (4.08, 0.05) -- (4.08,2);
			\node at (4.75, 0.4) {\textcolor{red}{$S(f^{(2)}_\theta)$}};

			\node at (4.5, -1.5) {\textcolor{black}{$f^{(2)}_\theta = \text{ReLU}(\theta_2\text{ReLU}(\theta_1))$}};
		\end{scope}
		\node[text width=5.9cm] at (9.0, 0.5) {In \textcolor{red}{red} are the points where
				the network map is \textcolor{red}{\emph{not} analytic
				($S(f_\theta)$)}. \\The complement of the
				\textcolor{green!60!black}{green} set
			are the points where the proof guarantees that the network map
		\textcolor{green!60!black}{\emph{is} analytic ($\R^2 \setminus
	B(f_\theta)$)}. We always have $ \textcolor{red}{S(f_\theta)} \subseteq
\textcolor{green!60!black}{B(f_\theta)}. $};
	\end{tikzpicture}
	\caption{Illustration of the main idea in the proof of
	Proposition~\ref{prop:ff-fcaea}.}
\end{figure}

\begin{proposition}[\bf Neural networks are a.e.\ analytic]
	\label{prop:nnaea}
	Let $ f_\theta:\R^{n_0} \to \R^{n_D} $ be an FF-FC, convolutional
	network or a network using softmax attention with piecewise
	analytic activation functions. Then for any input $ x \in \R^{n_0} $, the map
	$ \theta \mapsto f_\theta(x) $ is almost everywhere analytic.
\end{proposition}

\begin{proof}[Proof sketch.]
	The proof relies on Proposition~\ref{prop:ff-fcaea}. Let $ f_\theta$ be the
	network map. Then we can directly apply Proposition~\ref{prop:ff-fcaea} to
	show that the map $ \theta \mapsto f_\theta(x) $ is a.e.\ analytic for any
	input $ x \in \R^{n_0} $. We explain in Appendix~\ref{section:proof} how that can be done
	for the convolutional and softmax attention layers.
\end{proof}

The previous proposition is a key result, but it cannot be used directly to conclude that a
GD map is non-singular. Theorem~\ref{thm:main} will be proven in
the next section. We now show that the empirical loss from a
neural network is a.e.\ analytic. This is an important result since we need it
to  rove Theorem~\ref{thm:main}. We start with a lemma concerning
the composition of a.e.\ analytic maps and analytic maps.

\begin{lemma}[\bf Composition of a.e.\ analytic maps]
	\label{lem:composition}
	Let $ f:\R^m \to \R^n $ be an a.e.\ analytic function and $ g:\R^n \to \R^l $
	an analytic function. Then the composition $ g\circ f:\R^m \to \R^l $ is
	a.e.\ analytic. 
\end{lemma}

\begin{proof}
	We have that $ f|_{D(f)} $ is analytic. Applying the chain rule, we have
	that $ g\circ f|_{D(f)}  = (g \circ f)|_{D(f)}$ is analytic, hence $ g\circ
	f$ is analytic a.e.
\end{proof}

\begin{cor}
	Let $ f_\theta:\R^{n_0} \to \R^{n_D} $ be any neural network with input
	dimension $ n_0 $, output dimension $ n_D $, and $ n_\theta $ parameters,
	using piecewise analytic activation functions. Then, given any dataset $ \{
	(x_i, y_i)\}_{i=1}^m $ and any analytic loss function $
	l:\R^{n_D}\times\R^{n_D} \to \R $, the empirical loss given by
	\[
		L(\theta) = \frac{1}{m} \sum_{i=1}^n l(y_i, f_\theta(x_i))
	\]
	is almost everywhere analytic.
\end{cor}

\begin{proof}
	The maps $ \phi_i: \theta \mapsto (y_i, f_\theta(x_i))$ are a.e.\ analytic
	(since, by Proposition~\ref{prop:nnaea}, each component is), thus we can use Lemma~\ref{lem:composition} to
	conclude that the compositions $ l \circ \phi_i $ are a.e.\ analytic. The
	sum of a.e.\ analytic functions is a.e.\ analytic, thus it follows that $ L
	$ is a.e.\ analytic.
\end{proof}

\section{(Stochastic) Gradient Descent Map is Non-Singular}
\label{section:gd_map_non_sing}
We now use the results of the previous section to show that any gradient descent
map is non-singular. Since this is true for any number of data points, we can
immediately conclude that any SGD map $ G_{\eta, i} $ is also non-singular.
Because of this, we will state our results for GD, but keep in mind that they
also hold for SGD as well. In particular, observe that we have an explicit
description for any gradient descent map for almost all points. We start with a
helpful lemma.

\begin{lemma}[\bf Submersions are non-singular]
	Let $ f: \R^n \supset U \to \R^n $ be a smooth map. If the set of critical
	points of $ f $, i.e.\ points where the Jacobi determinant $ \det(Df) $ is 
	zero, has measure zero, then $ f $ is non-singular.
\end{lemma}

\begin{proof}
	Suppose that $\det(Df(x)) \neq 0$ for any $x \in U$. This, together with the
	inverse function theorem \citep[][Theorem 4.5]{lee2012} tells us that $f$ is
	a local diffeomorphism.

	Let $B \subset \R^n$ be any measure zero set. Because $f$ is a local
	diffeomorphism, around every point $x \in f^{-1}(B)$ there exists an open
	set $U_x \ni x$ such that $f$ restricted to it, $f|_{U_x}\colon U_x
	\rightarrow f(U_x)$, is a diffeomorphism. Since $f^{-1}(B) \subset \R^n$ and
	$\R^n$ is second countable, we can extract a countable subcover $\left\{
	U_{x_i} \right\}_{i=1}^\infty$ of $f^{-1}(B)$. Because $B$ has measure zero,
	so does each intersection $B_i = B \cap f(U_{x_i})$. Since $f$ restricted to
	$U_{x_i}$ is a diffeomorpishm, $f|_{U_i}^{-1}(B_i)$ has measure zero for all
	$i$. Observe that $$f^{-1}(B) \subset \cup_{i=1}^\infty
	f|_{U_i}^{-1}(B_i).$$ A countable union of measure zero sets has measure
	zero and a subset of a measure zero set has measure zero, thus the special
	case is proved.
 
	Going back to the general case, suppose that $ V = \{\theta|\det(Df(\theta))
	= 0\} $ has measure zero. The determinant is a continuous map and a
	singleton is closed, thus $V$ is closed. Then, the restriction of $f$ to the
	open set $U = \R^n\setminus V$ is a local diffeomorphism.

	Let $B \subset \R^n$ be any measure zero set. Its preimage under $f$ can be
	written as $$f^{-1}(B) = f|_{V}^{-1}(B) \cup f|_{U}^{-1}(B).$$ The first
	term in the union is the intersection of a set with $V$, which has measure
	zero by assumption, thus it has measure zero. The second term has measure
	zero using the special case proved above. It follows that $f^{-1}(B)$ has measure zero. 
\end{proof}

We use the previous lemma to show that \emph{the} gradient descent map for analytic
loss functions is non-singular. We emphasize that there is a \emph{unique} gradient
descent map if the loss is analytic. 

\begin{breakproposition}[\bf GD map for analytic loss functions is non-singular for
	almost all step-sizes]
	Let $ L:\R^{n_\theta} \supset U \to \R $ be an analytic function and assume without
	loss of generality that $ U $ is connected. Then, for almost all $ \eta > 0
	$, the GD map $ G_\eta: x \to x - \eta\nabla L(x) $ is non-singular.
\end{breakproposition} 

\begin{proof}
	If $ U $ has more than one connected component, we will show that $ G_\eta|_{U_i}
	$, where $ U_i $ is a connected component of $ U $, is non-singular. Since $ U
	$ has at most countably many connected components, it follows that $ G_\eta $ is
	non-singular as well.

	If all the eigenvalues $ \lambda_1, \ldots, \lambda_{n_\theta} $ of the
	$H_L$, the Hessian of $ L $, are constant, then for any $ \eta \not\in
	\{1/\lambda_1, \ldots, 1/\lambda_{n_\theta}\} $ we have that $ \det(I - \eta
	H_L) = \prod_{i=1}^{n_\theta} (1-\eta\lambda_i)$ is non-zero. This means
	that the Jacobi determinant $ \det DG_\eta = \det(I - \eta H_L) $ is always
	non-zero, hence $ G_\eta$ is non-singular.

	If at least one eigenvalue $ \lambda $ is not constant, there exist two
	points $ \theta_1, \theta_2 \in U $ such that $ \lambda(\theta_1) \neq
	\lambda(\theta_2) $.  Take any analytic path $\gamma\colon \R \to U$
	with $t_1\neq t_2\in \R$, $\gamma(t_1) = \theta_1$ and
	$\gamma(t_2) = \theta_2$. This can be done since we can interpolate any
	finite number of points using analytic functions. Since $\gamma$ and the
	map $\theta \mapsto H_L(\theta)$ are both analytic, their composition, $t
	\mapsto H_L\circ\gamma(t) \in \R^{{n_\theta}\times {n_\theta}}$, is analytic as well. In
	this case, the eigenvalues of $H_L\circ\gamma$ can be globally
	parameterized by analytic functions, i.e., there exist analytic functions
	$\lambda_i\circ\gamma\colon \R \to \R$ equal to the eigenvalues
	of $H_L\circ\gamma(t)$ for all $t$ \citep{kato1995}.

	For a non-constant analytic function, the set where it is equal to a
	constant $c\in \R$, $Z(\lambda_i\circ\gamma - c) \subset \R$
	has measure zero. For any $\overline{t} \in \R \setminus\cup_{i=1}^{n_\theta}
	Z(\lambda_i\circ\gamma) \subset \R$, it holds that
	$\lambda_i\circ\gamma(\overline{t}) \neq c$. In particular, for any $\eta>0$ there exists
	$\overline{\theta} = \gamma(\overline{t}) \in \R^{n_\theta}$ such that $\det(I
	- \eta H_L(\overline{\theta})) \neq 0$. We have showed that the analytic map
	$ \theta \mapsto \det(I-\eta H_L(\theta)) $ is not constant, thus its set of
	zeros has measure zero.  We conclude that $G_\eta$ is invertible almost
	everywhere, i.e., the set of points where the Jacobi determinant is zero has
	measure zero, whence we can apply the previous lemma to show that $ G_\eta $ is non-singular.
\end{proof}

\begin{cor}[\bf GD map is non-singular]
	\label{cor:gd_on_ae}
	Let $ L:\R^{n_\theta} \to \R $ be an a.e.\ analytic function, $ \eta > 0$,
	and $ G_\eta:\R^{n_\theta} \to \R^{n_\theta} $ be a function such that for
	all $ \theta \in D(L) $ we have that $ G_{\eta}(\theta) = \theta -
	\eta\nabla L(\theta) $. Then, for almost all $ \eta > 0 $, any gradient
	descent map $ G_\eta $ is non-singular.
\end{cor}

\begin{proof}
	By the previous proposition, $ G_\eta $ is non-singular when
	restricted to $ D(L) $ for almost all values of $ \eta $. If
	$ B $ is any measure zero set in $ \R^{n_\theta} $, we have that $ G_{ \eta}^{-1} (B)
	= G_{\eta}|_{D(L)}^{-1} (B) \cup G_{\eta}|_{S(L)}^{-1} (B)$. The first set
	in the union has measure zero since $ G_{\eta} $ is non-singular on $ D(L) $
	and the second one as well since $ S(L) $ has measure zero.
\end{proof}

\section{Discussion} \label{section:discussion}
In this paper, we have demonstrated that for neural networks with piecewise
analytic activation functions, the GD map $ G_\eta(\theta) = \theta - \eta\nabla
L(\theta) $ and the SGD maps $ G_{\eta, i} $ are non-singular for almost all
step-sizes $\eta $. This result has profound implications for the theoretical
understanding of neural network optimization. By proving that the loss function
is analytic almost everywhere and that the (S)GD map preserves sets of measure zero
under preimages, we confirm a key assumption in prior works, such as those by
\citet{lee2019} and \citet{chemnitz2024}. This bridges a critical gap between
theoretical guarantees and the practical success of GD in training deep
networks.

Our findings imply that, for almost all initializations and step-sizes, the
optimization trajectory avoids pathological behaviors --- such as convergence to
saddle points or unstable minima\footnote{One also has to show that the saddle
points and the unstable minima lie in a measure zero set. When $ n_\theta
\leq m $ this is true since such points are isolated. In the
overparameterized case, when $ n_\theta > m $, proving this requires more care.
\citet{cooper2020, cooper2021} have proven this for neural networks with smooth activations.}
--- aligning with empirical observations of (S)GD’s effectiveness. The proof
relies on two key insights: first, the piecewise analyticity of the activation
functions ensures the loss is analytic outside a set of measure zero; second,
the analytic properties of the loss in the ``nice'' regions ensure the GD map is
invertible almost everywhere, except for a negligible number of values for the
step-size.  While our result is robust across almost all step-sizes, we note
that non-singularity may fail at specific values for the step-size (where $ \eta
= 1/\lambda_i $ for an eigenvalue $ \lambda_i $ of the Hessian of $ L $), which
are negligible in practice.

To illustrate the connection between our result and the stability of minima, we
first recall a key result from \citet{chemnitz2024}. They show that for a
generic parametric model, if the GD map is non-singular and the set of global
minima $ M $ forms a manifold, then the stability of a minimum $ \theta \in M
$ is completely determined by a computable quantity. Our Theorem~\ref{thm:main}
provides the missing piece to rigorously apply this framework to neural
networks. By establishing the non-singularity of the (S)GD map for 
networks with piecewise analytic activations, we can now state the following:

\begin{cor}
	Let the assumptions be as in Theorem~\ref{thm:main}. Additionally, suppose
	the data is generic and let $ M $ be the set of global minima of the loss
	function $ L $. Then, for almost all step-sizes, there exist functions $
	\mu(\theta) $ and $ \lambda(\theta) $ such that a global minimum $ \theta
	\in M $ is stable under GD or SGD only if $ \mu(\theta) < 0 $ or $
	\lambda(\theta) < 0 $, respectively. Conversely, if $ \mu(\theta) > 0 $ or $
	\lambda(\theta) > 0 $, then $ \theta $ is unstable.
\end{cor}

The definitions of $ \mu $ and $ \lambda $ are the same as those in
\citet{chemnitz2024}. For a discussion of ``generic data'' see
\citet{cooper2021}; note that it is reasonable to assume that noisy data is
generic.

We now use this corollary to analyze the training dynamics in a few examples. We
first compute the functions $ \mu $ and $ \lambda $ for the simple setting with
loss function $ (\theta_1, \theta_2) \mapsto 3.51 (1 - \text{ReLU}
(\theta_2\text{ReLU}(\theta_1)))^2 $ (corresponding to a two-layer ReLU neural
network as described in Appendix~\ref{section:experiment}). We then use these
functions to: 1) determine the stability of periodic trajectories under GD as
the step-size $ \eta $ changes (Figure~\ref{fig:bif}), and 2) show that for a
fixed step-size, a global minimum can be stable under GD but unstable under SGD,
and vice-versa (Figure~\ref{fig:gd_sgd}). The next two paragraphs elaborate on this.

\par{\bf Application: periodic trajectories.}
The composition of non-singular maps is
non-singular as well, hence the stability analysis of minima can be extended to
periodic trajectories. The existence of such trajectories can be motivated
twofold: (i) the use of large step-sizes (as incentivized by results from
\citet{li2019, mohtashami2023}) leads to minima being stable only if they are
flat (cf.\ \citet{chemnitz2024});  (ii) it has been empirically observed that
the eigenvalues of the Hessian matrix of the loss
oscillate even after numerous training epochs \citep{cohen2022, cohen2022b},
which means that the trajectory is not converging. Our results provide a
comprehensive framework for dealing with such cases. We illustrate this through
an experiment described in Appendix~\ref{section:experiment}. We
assume a 2-layer ReLU network $ f_\theta(x) =
\text{ReLU}(\theta_2\text{ReLU}(\theta_1x)) $, a quadratic loss, and data
sampled from a linear function. For $ k>0 $, we are interested in $ k $-periodic points, that
is points $ \theta^* \in \R^2 $ such that $ G^k_\eta(\theta^*) = \theta^* $. The
sequence of iterates $ \theta^*, G_\eta(\theta^*), \ldots,
G^{k-1}_\eta(\theta^*) $ is then a periodic trajectory. A periodic
trajectory is stable if there is a region $ U $ around $ \theta^* $ such that
initializing in $ U $ leads to trajectories that stay close to the periodic one.
In Figure~\ref{fig:bif} on the right we plot trajectories with the same initialisation,
converging to different periodic trajectories for different step-sizes. In
Figure~\ref{fig:bif} on the left we plot the bifurcation diagram for periodic trajectories lying on the
diagonal ($ \theta_1 = \theta_2 $). As the step-size increases, new
periodic trajectories are created and the old ones become unstable.

\begin{figure}[ht!]
	\begin{tikzpicture}[scale=0.6,/tikz/background rectangle/.style={fill={rgb,1:red,1.0;green,1.0;blue,1.0}, fill opacity={1.0}, draw opacity={1.0}}, show background rectangle]
		\input{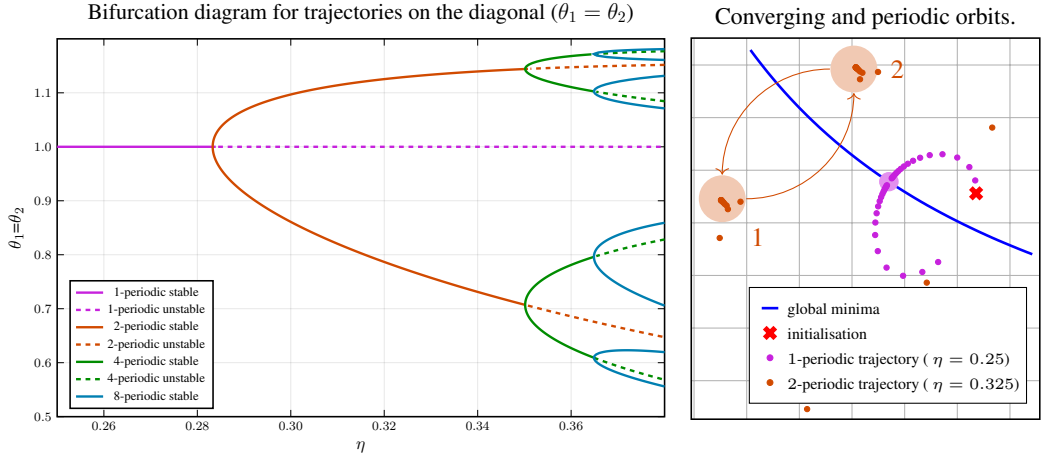}
	\end{tikzpicture}
	\tikzset{cross/.style={cross out, draw, minimum
	size=2*(#1-\pgflinewidth), inner sep=0pt, outer sep=0pt}}
	\begin{tikzpicture}
		\definecolor{violet}{rgb}{0.78, 0.129, 0.867};
		\definecolor{brown}{rgb}{0.82, 0.29, 0.0};
		\begin{scope}[xscale=5, yscale=4.93,local bounding box=traj]
			\draw[line width=1, domain=0.85:1.6,color=blue, samples=100] plot (\x, {1/\x)});

			\draw[color=violet!50!white, fill=violet!50!white] (1.218,0.82102) circle
				(0.7pt);

			\node[shape=circle,color=brown!30!white, fill=brown!30!white] (n1) at (0.775,0.775) {1};
			\node[shape=circle,color=brown!30!white, fill=brown!30!white] (n2)
				at (1.125,1.125) {1};
			\draw[color=brown, ->] (n1) to [bend right=45] (n2);
			\draw[color=brown, ->] (n2.west) to [bend left=-45] (n1.north);

			\node[below right=0.05 of n1] {\textcolor{brown}{1}};
			\node[right=0.05 of n2] {\textcolor{brown}{2}};
			\tikzmath{
				function gradx(\x, \y, \e) {
					if \x > 0 then { 
						if \y > 0 then {
							return \x+3.53*\e*\y*(1-\x*\y);
						} else {
							return \x;
						};
					} else {
						return \x;
					};
				};
				function grady(\x, \y, \e) {
					if \x > 0 then { 
						if \y > 0 then {
							return \y+3.53*\e*\x*(1-\x*\y);
						} else {
							return \y;
						};
					} else {
						return \y;
					};
				};
				real \xs, \ys, \xo, \yo, \xn, \yn, \e;
				int \k1, \k2;
				\k1 = 35;
				\k2 = 35;
				\xs = 1.45;
				\ys = 1/\xs+0.1;
				\e = 0.25;
				\xo = \xs;
				\yo = \ys;
				for \i in {1,2,...,\k1} {
					\xn = gradx(\xo, \yo, \e);
					\yn = grady(\xo, \yo, \e);
					{\draw[color=violet, fill=violet] (\xn, \yn) circle (0.2pt);};
					\xo = \xn;
					\yo = \yn;
				};
				\e = 0.325;
				\xo = \xs;
				\yo = \ys;
				for \i in {1,2,...,\k2} {
					\xn = gradx(\xo, \yo, \e);
					\yn = grady(\xo, \yo, \e);
					{\draw[color=brown, fill=brown] (\xn, \yn) circle (0.2pt);};
					\xo = \xn;
					\yo = \yn;
				};
				{\draw[color=red, fill=red, line width=2] (\xs, \ys) node[cross=4pt, red] {};};
			}
		\end{scope}

		\draw[line width=0.6] ($(traj.south west) - (0.1,0.1)$) rectangle 
			($(traj.north east) + (0.1,0.1)$);
		\begin{scope}[on background layer]
			\draw[thin, step=0.7,color=gray!60!white] ($(traj.south west) - (0.1,0.1)$) grid 
			($(traj.north east) + (0.1,0.1)$);
		\end{scope}

		\begin{scope}[scale=5.5, shift={(-0.41,-1.02)}]
			\draw[thin,fill=white,draw=black] (1.18,1.22) rectangle (1.85,1.5);
			\draw[line width=1, color=blue] (1.21, 1.45) -- (1.25, 1.45);
			\node[anchor=mid west] at (1.25, 1.45) {\tiny global minima};
			\draw[color=red, fill=red, line width=2] (1.23, 1.39) node[cross=4pt, red] {};
			\node[anchor=mid west] at (1.25, 1.39) {\tiny initialisation};
			\draw[color=violet, fill=violet] (1.23, 1.33) circle (0.2pt);
			\node[anchor=mid west] at (1.25, 1.33) {\tiny $1$-periodic
				trajectory ( $ \eta = 0.25$)};
			\draw[color=brown, fill=brown] (1.23, 1.27) circle (0.2pt);
			\node[anchor=mid west] at (1.25, 1.27) {\tiny $2$-periodic
				trajectory ( $ \eta = 0.325$)};
		\end{scope}

		\node[above=0.1 of traj.north] {\small Converging and periodic orbits.};
		\draw[color=white] ($(traj.south west) - (0.1,0.12)$) -- ($(traj.south west) -
			(0.1,0.76)$);
	\end{tikzpicture}
	\caption{Left: periodic trajectories for GD on $ L(\theta_1, \theta_2) = 3.53(1 -
		\text{ReLU}(\theta_2 \text{ReLU}(\theta_1)))^2$. As the step-size $ \eta $ increases,
		higher order orbits appear and the lower order ones become unstable. Right:
		for the same initialisation, but different $ \eta $, two trajectories, one that converges and one that oscillates. } 
	\label{fig:bif}
\end{figure}

\par{\bf Stochastic and adaptive variants of GD.} Our theoretical framework robustly
covers SGD and settings where a learning rate schedule is used.
Theorem~\ref{thm:main} establishes the non-singularity of the GD map $ G_{\eta,
i} $ for \emph{any fixed data and any batch size.} An SGD step is mathematically
equivalent to a full GD step performed on a loss function computed over a
mini-batch of data. Since our result holds \emph{irrespective} of the number of
data points used to compute the loss, it directly applies to each step of an SGD
trajectory, where the mini-batch changes at each iteration. The map for each
step is therefore non-singular (for almost all $ \eta $), and the composition of
these maps, which describes the full SGD trajectory, is also non-singular. An
optimisation process using GD over $ T $ steps with a learning rate schedule $
 (\eta_1, \eta_2, \ldots, \eta_T) $ can be described by the composite map $ G =
 G_{\eta_T} \circ \cdots \circ G_{\eta_1} $. This is a composition of
 non-singular maps, hence their composition is also guaranteed to be
 non-singular. Therefore, our analysis provides a rigorous foundation for
 understanding optimization dynamics not just for a fixed step-size, but for any
 practical scheme that relies on composing gradient descent steps.

\par{\bf Application: stable minima of GD vs SGD.}
Results from \citet{wu2018} have shown that GD and SGD can have different stable
global minima. For the setting described above, we can analytically compute the
quantities introduced in \citet{chemnitz2024} which determine stability (see
Appendix~\ref{section:experiment}). The relation between the minima of GD and
SGD can be quite complex, so we illustrate two interesting cases. In
Figure~\ref{fig:gd_sgd} on the left, we see that there are fewer stable global minima for SGD than
for GD. In Figure~\ref{fig:gd_sgd} on the right, we see a setting where the
stable minima for GD and SGD do not overlap at all. It remains an open problem
to determine if the stable minima for GD and SGD can be related in a general setting,
or if such a relation has to be determined on a case-by-case basis.

\begin{figure}[ht!]
	\begin{tikzpicture}
		\begin{scope}[line width=4]
			\begin{scope}[scale=6.32, shift={(-0.6,-0.6)}]
				\draw[step=0.1, thin, color=gray!60!white] (0.6,0.6) grid
					(1.55,1.55);
				\draw[->,thin,color=gray] (0.564,0.60) -- (1.6,0.60);
				\draw[->,thin,color=gray] (0.60,0.564) -- (0.60,1.6);
				\node at (1.1,0.55) {Stable minima for GD and SGD overlap.};

				\draw[domain=0.666:1.5,color=blue, samples=100] plot (\x, {1/\x)});

				\draw[domain=0.74:1.37, color=green!80!black, samples=100] plot (\x, {1/\x});

				\draw[densely dashed,domain=0.81:1.24, color=red, samples=100] plot (\x, {1/\x});

			\end{scope}
			\begin{scope}[scale=2.1, shift={(3.295,0)}]
				\draw[step=0.30095, thin, color=gray!60!white] (0.0,0.0) grid
					(2.9,2.9);
				\draw[->,thin,color=gray] (-0.1,0) -- (3.0095,0);
				\draw[->,thin,color=gray] (0,-0.1) -- (0,3.0095);
				\node at (1.5,-0.15) {Stable minima do not
					overlap.};

				\draw[domain=0.3703:2.7,color=blue, samples=100] plot (\x, {1/\x)});

				\draw[domain=0.83:1.22, color=green!80!black, samples=100] plot (\x, {1/\x});

				\draw[densely dashed, domain=0.455:0.7, color=red, samples=100] plot (\x, {1/\x});
				\draw[densely dashed,domain=1.43:2.19, color=red, samples=100] plot (\x, {1/\x});

				\begin{scope}[scale=3.01, shift={(-0.74,-0.51)}]
					\draw[thin,fill=white,draw=black] (1.18,1.28) rectangle (1.75,1.5);
					\draw[color=blue] (1.21, 1.45) -- (1.25, 1.45);
					\node[anchor=mid west] at (1.25, 1.45) {global minima};
					\draw[color=green!80!black] (1.21, 1.39) -- (1.25, 1.39);
					\node[anchor=mid west] at (1.25, 1.39) {GD stable minima};
					\draw[densely dashed, color=red] (1.21, 1.33) -- (1.25, 1.33);
					\node[anchor=mid west] at (1.25, 1.33) {SGD stable minima};
				\end{scope}
			\end{scope}
		\end{scope}
	\end{tikzpicture}
	\caption{Illustrating the different stable minima for GD and SGD. In blue
		are the global minima for $ L(\theta_1, \theta_2) = 3.53(1 -
		\text{ReLU}(\theta_2 \text{ReLU}(\theta_1)))^2$; in green are the minima
		stable for GD; in red those stable for SGD. Left: the stable minima
		for SGD are a proper subset of the stable minima for GD. Right: the stable
		minima for SGD and GD do not intersect. This change is due to different
	step-sizes ( $ \eta_{\text{left}} =  0.15$ vs $ \eta_{\text{right}} =0.3 $)
and different probabilities for generating the data (cf.\
Appendix~\ref{section:experiment}).}
	\label{fig:gd_sgd}
\end{figure}
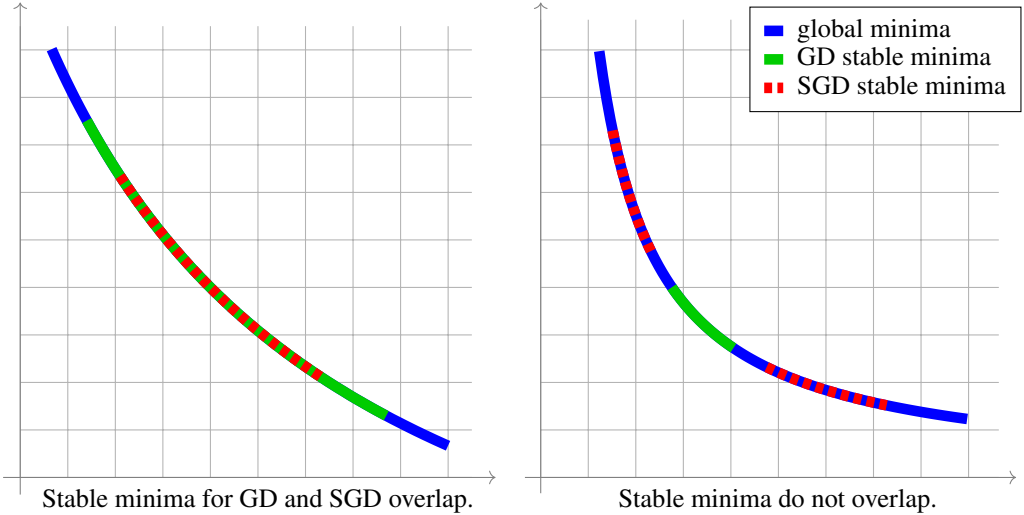

\par{\bf Different architectures.}
While our results already hold for a diverse number of layers (fully connected,
convolutional or using softmax attention) this does not cover all cases used in
practical settings. An in depth investigation of other architectures may reveal
additional structural properties that make it possible to prove that the neural
network map is analytic a.e.\ for those architectures as well. We think that an
extension to graph and residual neural networks might consist only of routine
applications of the ideas already present in our proofs.  However, for recurrent
neural networks, we see that our techniques are not enough and another approach
is needed. One might approach this setting by viewing the recurrent network $
f:\R^{n_0} \to \R^{n_D} $ as a larger fully connected network $ F:\R^N \to \R^{n_D} $ ($ N >
n_0 $) for which some weights are forced to be equal (cf.
Appendix~\ref{section:architectures}).  Proposition~\ref{prop:ff-fcaea}
guarantees that $ S(F) $, the singularities of $ F $ have measure zero in $ \R^N
$. We have that $ S(f) \subseteq \R^{n_0} \cap S(F) $ (where $ \R^{n_0} $ is viewed as a
subset of $ \R^N $). Hence, if $ \R^{n_0} \cap S(F) $ has measure zero, so must $
S(f) $ too. However, it is not clear that this intersection is negligible
and a more in depth investigation of $ S(F) $ is needed to be sure. Even with
this current limitation of our proof techniques, we conjecture that the GD map
is non-singular for the loss landscape of a recurrent neural network.

\par{\bf Other optimization algorithms.}
The techniques developed in this paper could be applied to other optimization
methods such as mirror descent or proximal point algorithms, potentially
broadening the scope of non-singularity guarantees. These alternative
optimization methods share with GD the core principle of
iteratively updating a solution to minimize a loss function (commonly using only
first-order information like the gradient).
\citet{lee2019} have analysed the convergence properties of such algorithms,
albeit in the same restricted setting they have for GD (i.e.\ requiring the loss
function to be Lipschitz smooth). To extend our results to
different optimization algorithms, one has to modify the proofs in
Section~\ref{section:gd_map_non_sing} of our work. In particular, it suffices to
investigate the case when the empirical loss function $ L $ is analytic, since
once it is proven that the optimizer map is non-singular for such losses, the
same proof as that for Corollary~\ref{cor:gd_on_ae} can be used to extend it to a.e.\
analytic losses.

\section*{Acknowledgements}
\label{section:acknowledgements}
This work has been partially supported by the German Research
Foundation (DFG) through the Priority Program SPP 2298 (project GH 257/2-2). We
also thank the reviewers for their feedback.

\bibliography{general}
\appendix

\newpage

\section{Different Architectures}
\label{section:architectures}
In this short appendix, we expand on the discussion from
Section~\ref{section:setting} regarding the different architectures besides the
ones made up only of fully connected layers.

\par{\bf Convolutional Networks.}
There exist also other kinds of neural network architectures, more specialized
to certain tasks. The first which we introduce are \emph{convolutional} neural
networks \cite{lecun1989}. This kind of networks are different from FF-FC
networks because they have convolutional layers. A detailed description of such
networks can be found in \cite{goodfellow2016}, but for our purposes it suffices
to know that a convolutional layer is such that some entries in the weight
matrix are equal to each other. For example, for a univariate convolution, each
row of the weight matrix is constrained to be equal to the row above shifted by
one element \cite{goodfellow2016}.

\par{\bf Recurrent Networks.}
To better learn time-series, \emph{recurrent} neural networks
\cite{rumelhart1986} have been introduced. The idea behind them is that if the
input to the network is a series $ x = (x_1, \ldots, x_T) $, the network might
benefit by using the intermediate activations it has computed at previous steps
when computing the current activation. More explicitly, if we denote by $
f^{(d)}_t(x_1, \ldots, x_t) $ the output of the $ d $-th layer of the network at
time $ t $, then the output of the $ (d+1) $-th layer of the network at time $
t+1 $ is defined recursively by the following relation
\[
       f^{ (d+1)}_{t+1}(x_1, \ldots, x_{t+1}) = \sigma(W_{d+1}f^{(d)}_{t+1}(x_1,
       \ldots, x_{t+1}) + U_{t+1}f^{ (d+1)}_{t} (x_1, \ldots, x_t) + b_{d+1}),
\]
where $ \sigma $ is applied entry-wise and the base case is $ f_1^{(1)} (x_1) =
\sigma(W_1 f^{ (0)}_0 + U_1 x_1 + b_1)$, with $ f^{ (0)}_0 $ a constant. The new
matrices we introduced $ U_i \in \R^{n_i} \times \R^{n_i} $ represent the recurrent
connections.

In this case, the outputs of the network at any time-step can be used when computing the
empirical loss function. That is, for every input $ x_i $ (we index
different inputs with the subscript $ i $), we have a set $ I_i \subset [T]$ of time
instances, and we can write the empirical loss as
\[
       L(\theta) = \frac{1}{m} \sum_{i=1}^m \frac{1}{|I_i|} \sum_{t\in I_i} l(y_i, f^{ (D)}_t(x_{i,1},
       \ldots, x_{i, t})).
\]

\par{\bf Networks using Attention.}
As an alternative to recurrent networks, \emph{attention} layers have been
developed~\citep{vaswani2017}. Such a layer has three learnable matrices $ W_q,
W_k, $ and $ W_v $, and works as follows: given a sequence $ x_1, \ldots, x_n
\in \R^n_0$, it outputs a sequence $ y_1, \ldots, y_n \in \R^n_D$. Denoting
by $ X $ the input matrix, i.e.\ the matrix having the $ x_i $'s as columns, we
obtain three matrices: $ Q = W_qX, K = W_kX, V=W_vX $ called the queries, keys,
and values, respectively. The output is: 
\[ 
	Attention(Q, K, V) = 	softmax(\frac{QK^T}{\sqrt{n}} )V, 
\] 
where the $ softmax $ function is independently applied to every row of its argument.
The columns of this matrix are the output sequence $ y_1, \ldots, y_n $. After
applying a number of attention layers, we may use an FF-FC network at the end.
This is usually done by concatenating the output sequence into a vector $ y =
y_1 * \ldots * y_n $ and using this as the input to an FF-FC network. 
The important result for our analysis is to observe that for any
input sequence $ x_1, \ldots, x_n $, the map determined by an attention layer on
the parameters $ (W_q, W_k, W_v) \mapsto Attention(Q, K, V) $ is
\emph{analytic}.

\section{Proofs of Main Results}
\label{section:proof}
In this appendix we provide the full proofs of Propositions~\ref{prop:ff-fcaea}
and~\ref{prop:nnaea}. There is one difference from the result stated in the main
body of the text. Instead of proving the result for a.e.\ analyticity, we prove
a more general result, showing that the smoothness properties of the activation
function carry over to the neural network map almost everywhere. What we mean by
this is that if the activation function is piecewise $ \mathcal{C}^k $, where $
k $ can be a positive integer (denoting $ k $-times differentiable functions), $
\infty $ (denoting smooth functions), or $ \omega $ (denoting analytic
functions), then the neural network function is also almost everywhere $
\mathcal{C}^k $. The definition of an almost everywhere $ \mathcal{C}^k $
function is analogous to Definition~\ref{def:aea}. We begin by restating 
Proposition~\ref{prop:ff-fcaea}. 

\begin{proposition}[\bf Analogue of chain rule for fully connected layers]
	\label{prop:chnn}
	Let $ D > 0 $ be a positive number and $ \{\sigma_i:\R^{n_i} \to
	\R^{n_i}\}_{i=1}^D $ be a collection of a.e. $ \mathcal{C}^k $ maps with the property
	that $ \partial Z(\sigma_i) $ has measure zero and $ \alpha \in \R^{n_0} $ a
	vector. Then the map $ f_D $ defined recursively by
	\begin{align*}
		f_d:\R^{n_1\times n_0}\times \cdots \times\R^{n_d\times n_{d-1}} &\to
		\R^{n_d}\\
		(W_1, \ldots, W_d) &\mapsto \sigma_d(W_df_{d-1}(W_1, \ldots, W_{d-1})),
	\end{align*}
	with $ f_1(W_1) = \sigma_1(W_1\alpha)$ is $ \mathcal{C}^k $ almost everywhere and the set $
	\partial Z(f_D) $ has Lebesgue measure zero.
\end{proposition}

\begin{proof}
	We will prove this proposition by induction on $ D $.  If $ D=1 $,  we have
	that $ f_1(W_1) = \sigma_1 (W_1\alpha) $.

	If $ \alpha = 0 $, then $ f_1 (W_1) = \sigma_1 (0) $ is a constant function,
	hence it is $ \mathcal{C}^k $.  Also, we have that $ Z(f_1) $ is either the whole
	domain of $ f_1 $ or the empty set, hence $ \partial Z(f_1) $ has measure
	zero.

	In the case $ \alpha\neq 0 $, the set of points $ S(f_1) = \{W_1\alpha \in
	S(\sigma_1)\} $ (\emph{a priori} this set need not be equal to the set of points where $
	f_1$ is \emph{not} $ \mathcal{C}^k $; we show in this paragraph that they
	\emph{do} coincide) has measure zero and is also closed.  This follows since the
	multiplication map $ M_1: W_1 \mapsto W_1\alpha $ is non-singular. With
	this definition, we have that $ f_1 $ is the composition $ f_1 = \sigma_1 \circ
	M_1 $. Taking any $ W_1 \not\in S(f_1) $, we have that $ M_1 $ is $
	\mathcal{C}^k $ in a neighborhood of $ W_1 $ and that $ \sigma_1 $ is $ \mathcal{C}^k $ in a
	neighborhood of $ W_1\alpha = M_1(W_1) $, hence we can apply the chain
	rule to conclude that $ f_1 $ is $ \mathcal{C}^k $ in a neighborhood of $ W_1 $. We
	see that $ f_1 $ is $ \mathcal{C}^k $ in the complement of $ S(f_1) $. But $ S(f_1)
	$ is a closed set and has measure zero, hence $ f_1 $ is $ \mathcal{C}^k $ almost
	everywhere.

	We now show that $ \partial Z(f_1) $ has measure zero. Looking at the set $
	Z(f_1) = \{W_1 | W_1\alpha \in Z(\sigma_1)\}$, we see that $ Z(f_1) =
	M_1^{-1}(Z(\sigma_1)) $. We also have that $ W \in \partial Z(f_1) $ implies
	that $ M_1(W) \in \partial Z(\sigma_1) $. Hence, we see that $ \partial
	Z(f_1) \subset M_1^{-1}(\partial Z(\sigma_1)) $. Since $ M_1 $ is a
	non-singular map (it is a submersion), it follows that $ \partial
	Z(f_1) $ has measure zero.

	Now we proceed to the induction step. Assume that $ f_{D-1} $ is $ \mathcal{C}^k $ a.e.\ and
	that $ \partial Z(f_{D-1}) $ has measure zero. Define the following sets:
	\begin{enumerate}
		\item
			The set of ``bad'' points of $ f_{D-1} $:
			\[
				B(f_{D-1}) = \partial Z(f_{D-1}) \cup S(f_{D-1}).
			\]
			By assumption, we have that $ B(f_{D-1}) $ is a measure zero set.
		\item
			The set of ``nice'' points of $ f_{D-1} $ which are not roots:
			\[
				N(f_{D-1}) = \text{dom}(f_{D-1})\setminus \left(B(f_{D-1}) \cup
				\text{int}(Z(f_{D-1}))\right),
			\]
			where $ \text{dom}(f_{D-1}) $ is the domain of $ f_{D-1} $.
	\end{enumerate}

	We also use the following notation
	\[
		x_D = (W_1, \ldots, W_D)  \text{ and } y_{D-1} = W_Df_{D-1} (W_1, \ldots, W_{D-1}).
	\]

	We now consider two cases:
	\begin{enumerate}
		\item \label{item:g1}
			If $ x_{D-1} \in N(f_{D-1}) $, then take any $ W_D $
			such that $ W_Df_{D-1}(x_{D-1}) \not\in S(\sigma_D) $. Such
			a $ W_D $ exists since $ f_{D-1}(x_{D-1}) $ is not
			zero. In fact, the complement of the set of all such points has
			measure zero since it is the preimage of the set $ S(\sigma_D) $ under
			the multiplication map $ M_D: W_D \mapsto W_Df_{D-1}(x_{D-1})$,
			which is a non-singular map.

			We will write $ f_D $ as a composition of maps and then show that we
			can use the chain rule to conclude that $ f_D $ is $ \mathcal{C}^k $ in a
			neighborhood of $ x_D $. Thus, consider the following maps:
			\begin{align*}
				f_{D-1} \times \id: \R^{n_1\times n_0}\times \ldots \times
				\R^{n_D\times n_{D-1}} &\to \R^{n_{D-1}}\times \R^{n_D\times
				n_{D-1}}\\
					(W_1, \ldots, W_D) &\mapsto (f_{D-1} (W_1, \ldots, W_{D-1}),
					W_D)
			\end{align*}
			\begin{align*}
				M_D: \R^{n_{D-1}} \times \R^{n_D\times n_{D-1}} &\to \R^{n_D}\\
				(x_{D-1}, W_D) &\mapsto W_Dx_{D-1}.
			\end{align*}

			We see that $ f_{D} $ can be written as the composition $ \sigma_D
			\circ M_D \circ (f_{D-1}\times \id)$. By assumption, we have that $
			f_{D-1}\times \id$ is $ \mathcal{C}^k $ in a neighborhood of $ x_D $; $ M_D $
			is $ \mathcal{C}^k $ in its whole domain, in particular in an open
			neighborhood of $ (f_{D-1}(x_{D-1}), W_D) $; and also $ \sigma_D $ is $
			\mathcal{C}^k $ on an open neighborhood of $ y_D $ by assumption. Hence, we can
			apply the chain rule to conclude that $ f_D $ is $ \mathcal{C}^k $ in an open
			neighborhood of $ x_D $.

		\item \label{item:g2}
			If $ x_{D-1} \in \text{int}(Z(f_{D-1}))$, then we can use a
			different approach. Since it is an interior point, there exists an
			open neighborhood $ U \ni x_{D-1} $ such that $ U \subset
			\text{int} (Z(f_{D-1})) $. We immediately see that on this
			neighborhood $ f_D $ is constant, hence it is $ \mathcal{C}^k $.
	\end{enumerate}

	We can now describe the set of points where the previous two cases guarantee
	that $ f_D $ is differentiable. Start by defining $ \Delta $ to be the set
	\[
		\Delta = \{ (x_{D-1}, W_D) \in N(f_{D-1}) \times \R^{n_D\times n_{D-1}} |
		W_Df_{D-1} (x_{D-1}) \in S(\sigma_D)\}.
	\]
	It is closed since $ S(\sigma_D) $ is closed and the map $ M_D:(x_{D-1}, W_D) \mapsto
	W_Df_{D-1} (x_{D-1})$ is continuous. Since this map is measurable and so is
	$ S(\sigma_D) $, it follows that $ \Delta $ is also measurable. In the first
	case above, we
	proved that for any $ x_{D-1} \in N(f_{D-1}) $ the intersection $ \Delta \cap
	\{x_{D-1}\} \times \R^{n_D \times n_{D-1}}$ has measure zero. Use can use a
	special case of the Fubini's theorem to conclude that $ \Delta $ has measure zero as well.

	In case~\ref{item:g1} we prove that $ f_D $ is $ \mathcal{C}^k $ on $
	(N(f_{D-1}) \times \R^{n_D \times n_{D-1}}) \setminus \Delta $. In
	case~\ref{item:g2}, we showed that $ f_D $ is $ \mathcal{C}^k $ on $
	\text{int} (Z(f_{D-1}))\times \R^{n_D \times n_{D-1}} $. Putting everything
	together, we see that $ F $ is $ \mathcal{C}^k $ on an open set whose
	complement has measure zero, namely $ \Delta \cup B(f_{D-1}) \times  \R^{n_D
	\times n_{D-1}}$. This shows that $ f_D $ is $ \mathcal{C}^k $ almost
	everywhere.

	We now prove the second part, namely that $ \partial Z(f_{D}) $ has measure
	zero. We have that $ Z(f_D)  = \{ (x_{D-1}, W_D) | W_Df_{D-1} (x_{D-1}) \in
	Z(\sigma_D)\} $. Start by observing that if $ x_D \in \partial Z(f_D) $, then $
	y_D \in \partial Z(\sigma_D)$. This follows since $ M_D $ is a continuous map.
	This means that $ \partial Z(f_D) \subset M_D^{-1}(\partial Z(\sigma_D)) $. We
	consider three cases:
	\begin{enumerate}
		\item
			If $ x_{D-1} \in  N(f_{D-1}) $, we can look at the derivative of $
			M_D $ at a point $ (x_{D-1}, W_D) \in N(f_{D-1}) \times \R^{n_D
			\times n_{D-1}} $.  For any $ ( \bar{x}_{D-1}, \bar{W}_D) $ in the
			tangent space at $ x_D $, we have:
			\[
				d_{x_D}M( \bar{x}_{D-1}, \bar{W}_D) =
				\bar{W}_Df_{D-1}(x_{D-1}) + W_Dd_{x_{D-1}}f_{D-1}(
				\bar{x}_{D-1}).
			\]
			It is surjective since $ f_{D-1} (x_{D-1}) \neq 0 $. Hence, $ M $ is
			non-singular on $ N(f_{D-1}) \times \R^{n_D \times n_{D-1}}$. Since
			$ \partial Z(\sigma_D) $ has measure zero, we have that $ N(F) \times
			\R^{n_D \times n_{D-1}} \cap \partial Z(f_D) $ also has measure
			zero.

		\item
			If $ x_{D-1} \in \text{int}( Z(f_{D-1})) $, then $ W_Df_{D-1}
			(x_{D-1}) $ is constant for any $ W_D $. Thus, either all points $
			y_{D-1} $ are roots of $ \sigma_D $ or none. In both cases the set $
			\text{int}(Z(f_{D-1})) \times \R^{n_D \times n_{D-1}} \cap \partial Z(f_D) $ 
			has measure zero.

		\item
			If $ x_{D-1} \in B(f_{D-1}) $, then the set $ B(f_{D-1}) \times
			\R^{n_D \times n_{D-1}}$ has measure zero, hence its intersection
			with $ \partial Z(f_D) $ also has measure zero.

	\end{enumerate}
	Since the domain of $ f_D $ is a disjoint union of the three sets we have
	considered in the previous cases and since the intersection of $
	\partial Z(f_D) $ with each one of them has measure zero, it follows that $
	\partial Z(f_D)$ has measure zero itself. This concludes the proof.
\end{proof}

\begin{remark}[\bf Why the bias can be dropped]
	\label{rem:drop-bias}
	Including bias terms changes the linear transformation in each layer to an
	affine one: from $ W\alpha $ to $ W\alpha + b $. The proof of non-singularity relies on this map being a
	submersion (i.e., its derivative being surjective). The derivative of the affine map is
	surjective, just as it is for the map without bias. In fact, the presence of the
	additive bias term makes the surjectivity argument even more direct. The rest of
	the inductive proof in Proposition~\ref{prop:chnn}  proceeds in exactly the same manner.
\end{remark}

\begin{proposition}[\bf Analogue of chain rule for convolutional layers]
	\label{prop:conv}
	Let all the assumptions and definitions in Proposition~\ref{prop:chnn} hold
	true. However, suppose that some entries in the matrix $ W_D $ are forced to
	be equal to each other. Then, for generic data, the map $ F_D $ is a.e.\ $ \mathcal{C}^k $
	and the set $ \partial Z(F_D) $ has measure zero.
\end{proposition}

\begin{proof}
	Let $ C_D $ be the space spanned by the weight matrices $ W_D $. We must
	consider now the multiplication map as a map
	\[ 
		M_c: W_D \to W_D\alpha, \text{ with domain } C_D.
	\]

	We now show that $ M_c $ is non-singular. Observe that for any measure zero
	set $ B $ in the codomain, the preimage is $ M_c^{-1}(B) = M^{-1}(B) \cap
	C_D $, where $ M $ is the multiplication map on the full ambient space $
	R^{n_D \times n_{D-1}} $. Since $ M $ is a submersion (this was shown in the
	proof of Proposition~\ref{prop:chnn}), $ M^{-1}(B) $ has measure zero. The
	intersection of a measure zero set with a lower-dimensional subspace $ C_D $
	will also have measure zero within $ C_D $ for a generic choice of data.
	Such a choice is justified if the data is assumed to be noisy. While a
	pathological choice of data could align $ C_D $ with $ M^{-1} (B) $, such
	data points form a measure zero set themselves. 

	From here on, the proof of Proposition~\ref{prop:chnn} applies verbatim. 
\end{proof}

\begin{proposition}[\bf Analogue of chain rule for softmax attention]
	\label{prop:attention}
	Let $ \{ f_i:\R^{n_\theta} \to \R^{n_{D-1}} \}_{i=1}^{m}$ be a collection of
	a.e.\ $ \mathcal{C}^k $ maps with $ \partial Z(f_i) $ having
	measure zero for every $ i $. Let $ Y(\theta) $ be
	the matrix with $ f_i(\theta) $'s as columns. Then, the map 
	\[ 
		F: (\theta, W_k, W_q, W_v) \mapsto Attention(W_kY(\theta), W_qY(\theta),
		W_vY(\theta)) 
	\]
	is a.e.\ $ \mathcal{C}^k $ and the set $ \partial Z(F) $ has measure zero.
\end{proposition}

In the proposition above, one should think of the $ f_i $'s as representing the
output of a \emph{fixed} neural network as a function of the networks' parameters $
\theta $ for \emph{different} input values $ x_i $ in a sequence to be fed into
the attention layer. Usually, there is no pre-processing of the data before it
goes into the attention layer, however stating the proposition this way will
allow us more flexibility, for example, we can apply it for two attention layers
one after another.

\begin{proof}[Proof of Proposition~\ref{prop:attention}]
	If there is are no $ f_i $'s to be applied to the data before entering the
	attention layer (as is usually the case), the argument becomes simpler, and
	we also include it as a simplified version of the general proof.

	In this simplified case, we only have to show that the attention layer 
	$ A: (W_k, W_q, W_v) \mapsto Attention(W_kX, W_qX, W_vX) $ for a fixed input
	$ X $ has the desired properties. Since matrix multiplication, the softmax
	function, and concatenation are analytic, then their composition is
	definitely a.e.\ $ \mathcal{C}^k $. 

	To show that $ \partial Z(F) $ has measure zero, we need the following lemma.

	\begin{lemma}[\bf The zero set of analytic functions has measure zero]
		\label{lem:measure_zero_set_analytic}
		Let $ f\colon \R^n \rightarrow \R^m $ be a non-zero real analytic function. Then the zero
		set of $f$, $Z(f)$, has measure zero.
	\end{lemma}
	\begin{proof}
		The proof proceeds by induction on $n$. Let $ f \colon \R \rightarrow \R^m $ be a real
		analytic function. Suppose that $Z(f)$ is not a set having measure zero. Since it is not a
		null-set, it is an uncountable set, thus it has an accumulation
		point. From the identity theorem \citep[Corollary 1.2.7]{krantz2002} we
		get that $f \equiv 0$, contradicting our assumption that $f$ is non-zero.

		Suppose the lemma is established for real analytic functions defined on
		$\R^{n-1}$. Assume that $Z(f)$ is not a null-set. By Fubini's Theorem
		there is a set $E \subset \R$ not having measure zero such that for all $x
		\in E$ the set $Z(f)\cap (\{x\}\times \R^{n-1})$ does not have measure zero. By the
		induction hypothesis we conclude that $f = 0$ on each of those
		hyperplanes. Since $E$ is uncountable, it has an accumulation point
		$a \in \R$. Thus, there is a sequence of distinct points $a_k \mapsto a$ such
		that $f = 0$ on the hyperplanes with first coordinate in $\left\{ a_1, a_2, \cdots
		\right\} \cup \left\{ a \right\}$.

		Now take any line of the form $\R \times \{x_2\} \times \cdots \times \{x_n\}$,
		where $x_i \in \R$ are fixed numbers. Since $f$ is real analytic on this
		line and $f = 0$ on a set with an accumulation point on that line, we have
		$f = 0$ on that line. This was for true any line, hence $f \equiv 0$ on
		$\R^n$.
	\end{proof}

	With this lemma, we have that $ Z(F) $ has measure zero, since attention is
	analytic and non-constant, thus $ \partial Z(F) \subset Z(F) $ also has
	measure zero.

	For the general case, where we apply the $ f_i $'s to the input of the
	attention layer, we proceed as follows.

	Let $ S(f) = \bigcup_{i=1}^{m} S(f_i) $. We have that for any $ \theta \in
	\R^{n_\theta} $ and $ (W_k, W_q, W_v) $, the function $ F $ is $
	\mathcal{C}^k $ since $ Y $ is $ \mathcal{C}^k $ in a neighborhood of $
	\theta $. Hence, we have that $ S(F) \subset S(f) \times \R^{v \times
	n_{D-1}} \times \R^{v \times n_{D-1}} \times \R^{v \times n_{D-1}}$, that
	is, $ F $ is $ \mathcal{C}^k $ on a dense open subset of its domain.

	We write $ F $ as the following composition of functions
	\[ 
		(\theta, W_k, W_q, W_v) \xmapsto{f} (Y(\theta), W_k, W_q, W_v) \xmapsto{M}
		 (W_kY, W_qY, W_vY) \xmapsto{A}
		 F(\theta, W_k, W_q, W_v).
	\]
	
	Using the Lemma~\ref{lem:measure_zero_set_analytic}, we see that $ Z(A) $ has measure zero, where $ A $
	is the map $ (K, Q, V) \mapsto Attention(K, Q, V) $. Since $ M $, the matrix
	multiplication, is a submersion, we have that $ M^{-1}(Z(A)) $ also has
	measure zero. In fact, since the preimages are hyperplanes in the matrix
	space, we have that $ M^{-1}(Z(A)) \cap (\theta, \R^{v \times n_{D-1}}, W_q, W_v) $ 
	has measure zero and similarly for $ W_q $ and $ W_v $. This means that $ f^{-1} (M^{-1} (Z
	 (A))) \supset \partial Z(F) $ will also have measure zero. 
\end{proof}

\begin{remark}[\bf Any order for the layers]
	In Propositions~\ref{prop:conv} and~\ref{prop:attention} we have proven the
	assumptions needed for the induction step in the proof of
	Proposition~\ref{prop:chnn}. This means that our results, namely the almost
	everywhere $ k $-times differentiability of the network function hold for
	any combination of convolutional, attention, or fully connected layers.
\end{remark}

\section{Experimental Details}
\label{section:experiment}
The experiment described in Section~\ref{section:discussion} is as follows. We
have two data points: $ (0.9,0.9) $ and $ (2.5,2.5) $ which determine a linear
function. The model we use is the following: $ f_\theta(x) =
\text{ReLU}(\theta_2\text{ReLU}(\theta_1x))$. With the quadratic loss, we get
the following empirical loss function
\[ 
	L(\theta_1, \theta_2) = 3.53(1 -
	\text{ReLU}(\theta_2\text{ReLU}(\theta_1)))^2.
\]

Then, the global minima for this loss are the points $ \{ (\theta_1, \theta_2) |
\theta_1\theta_2 = 1 \text{ and }  \theta_1, \theta_2 >0 \} $, i.e.\ the first
quadrant part of the hyperbola in the plane. We denote by $ p $ the probability
of selecting the point $ (0.9,0.9) $ as a mini batch for SGD; by $ \eta $
we denote the step-size.

\citet{chemnitz2024} introduce two quantities, $ \mu $ and $ \lambda $ which
determine if a global minimum $ \theta^* $ is stable for GD and SGD,
respectively. More precisely, we have that $ \theta^* $ is stable for GD if $
\mu(\theta^*) < 0 $ and that it is stable for SGD if $ \lambda(\theta^*) < 0 $.
For the setting we have described above, the analytic form of $ \mu $ and $
\lambda $ is as follows
\begin{align}
	\mu(\theta) &= \log(|1 - \eta(p0.9^2 + (1-p)2.5^2)(\theta_1^2 + \theta_2^2) |),\\
	\lambda(\theta) &= p\log(|1 - \eta0.9^2 (\theta_1^2+\theta_2^2)|) + (1-p)\log (|1 - 2.5^2
	 (\theta_1^2 + \theta_2^2)|).
\end{align}

Figure~\ref{fig:gd_sgd} was obtained by setting the $ \eta = 0.15 $ and $ p =
0.5 $ for the figure on the left and $ \eta = 0.3 $ and $ p = 0.58 $ for the figure on
the right.

To obtain the left part of Figure~\ref{fig:bif}, we look at the polynomial $
P^k(x) = G_\eta^{(k)}(x) - x $, and use a numerical root finding algorithm to find
all the real roots. We determine if a periodic point of period $ k $, $ \theta^k
$, is stable by numerically computing the eigenvalues of the linearization of
the $ k $-times iterated GD map at $ \theta^k $, $ DG^{(k)}(\theta^k) $. If both
eigenvalues have absolute value smaller than 1, we know \citep{chemnitz2024}
that the period is stable; otherwise it is unstable.

The right part of Figure~\ref{fig:bif}, plots two trajectories of GD initialised
at $ (1.48, 1/1.48 + 0.1) $ with step-sizes $ \eta = 0.25 $ for the trajectory
converging to a minimum (violet) and $ \eta = 0.325 $ for the trajectory
converging to a periodic orbit of period $2$ (brown).

\newpage

\end{document}